\documentclass[12pt]{amsart}
\usepackage{amsmath}
\usepackage{amssymb}
\usepackage[mathcal]{eucal}
\usepackage[all]{xy}
\usepackage{latexsym}
\usepackage{amstext}
\usepackage{amsfonts}
\usepackage{amsthm}
\usepackage{amsopn}
\usepackage{amsbsy}
\usepackage{layout}
\usepackage{color}
\usepackage{graphicx}
\usepackage{bm}
\usepackage{yfonts}
\usepackage{pifont}
\usepackage[scr]{rsfso}
\usepackage{tikz}
\newtheorem{thm}{Theorem}[section]
\newtheorem{lem}[thm]{Lemma}

\newtheorem{cor}[thm]{Corollary}
\newtheorem{pro}[thm]{Proposition}
\newtheorem{ex}[thm]{Example}
\newtheorem{que}[thm]{Question}
\newtheorem{prob}[thm]{Problem}

\newtheorem{df}[thm]{Definition}
\newtheorem{rem}[thm]{Remark}
\newtheorem{cl}[thm]{Claim}

\newtheorem{claim}{Claim}

\newcommand{\Int}{\mbox{{\rm Int}}}

\newcommand*\circled[1]{\tikz[baseline=(char.base)]{
            \node[shape=circle,draw,inner sep=1pt] (char) {#1};}}

\def\J{\mathbb J}

\def\N{\mathbb N}
\def\R{\mathbb R}

\def\supp{{\mathrm {supp}}\,}
\def\cl{\mathop{\it cl}}
\makeatletter
\def\Int{\mathop{\operator@font Int}\nolimits}
\makeatother

\begin{document}

\title[On densely defined linear continuous operators between function spaces]
{On densely defined linear continuous operators between function spaces}

\author{A. Leiderman and V. Valov}
\address{Department of Mathematics, Ben-Gurion University of the Negev,
Beer-Sheva, Israel}
\email{arkady@math.bgu.ac.il}

\address{Department of Computer Science and Mathematics, Nipissing University,
100 College Drive, P.O. Box 5002, North Bay, ON, P1B 8L7, Canada}
\email{veskov@nipissingu.ca}
\thanks{The second author was partially supported by NSERC Grant RGPIN-2025-07173.}



\begin{abstract}
For any Tychonoff space $X$, let $D(X)$ denote either the space $C(X)$ of all continuous real-valued functions on $X$ or the space $C^*(X)$ of all bounded continuous real-valued functions on $X$. 
We write $D_p(X)$ when $D(X)$ is endowed with the topology of pointwise convergence.

In our recently published paper \cite[Theorem 1.6]{elv}, we obtained the following result: if $T: D_{p}(X) \to D_{p}(Y)$ is a linear continuous surjection, where $X$ is a metrizable space and $Y$ is a perfectly normal space, then $Y$ inherits a given topological property $\mathcal{P}$ from $X$. 

A linear continuous surjection $T: E_{p}(X)\to E_{p}(Y)$ is said to be densely defined if $E(X)$ and $E(Y)$ are dense linear subspaces of $D_{p}(X)$ and $D_{p}(Y)$, respectively, and $E(X)$ is correct
 (see Definition \ref{df-correct}(a)).
 In the present paper, we establish sufficient conditions under which the above statement remains valid for a densely defined linear continuous surjection $T: E_{p}(X) \to E_{p}(Y)$.
In particular, $\mathcal{P}$ can be zero-dimensionality, $\sigma$-compactness or strong countable-dimensionality. 

Additionally, for arbitrary Tychonoff spaces $X$ and $Y$, assuming that $T: E_p(X)\to E_p(Y)$ is a densely defined linear continuous operator,
 we show that $X\in\mathcal P$ implies $Y\in\mathcal P$, where  $\mathcal P$ is the property $(\kappa)$, the strong $\sigma$-scatteredness,
or the property of being a $\Delta_1$-space.
\end{abstract}

\maketitle\markboth{A. Leiderman, V. Valov}{Densely defined linear continuous operators}




\section{Introduction and Main Results}\label{intro}
For a Tychonoff space $X$,
 we denote by $C(X)$ the linear space of all continuous real-valued functions on $X$.
$C^*(X)$ is the subspace of $C(X)$ consisting of the bounded functions.
We write $C_p(X)$ (resp., $C_p^*(X)$) if $C(X)$ (resp., $C^*(X)$) is endowed with the pointwise convergence topology.
A fundamental problem that has been intensively studied by many authors is the following:

\begin{prob} \label{pr0}
Which topological properties of the space $X$ are preserved by linear continuous (uniformly continuous) mappings / linear homeomorphisms (uniform homeomorphisms)
of spaces $C_p(X)$ or $C_p^*(X)$?
\end{prob}

Let us briefly mention a few crucial results.
\begin{itemize}
\item Let $T: C_p(X) \to C_p(Y)$ be a linear (even uniform, see \cite{gu}) homeomorphism. Then $\dim(X) = \dim(Y)$ (see \cite{p}).
\item For every finite-dimensional metrizable compact space $X$ there exists a linear continuous surjection $T: C_p([0,1]) \to C_p(X)$ (see \cite{lmp}).
\item Let $T: C_p(X) \to C_p(Y)$ (or $T: C_p^*(X) \to C_p^*(Y)$) be a linear continuous surjection. If $\dim(X)=0$ then $\dim(Y)=0$ (see \cite{ev}, and \cite{KawL} for compact $X$, $Y$).
\item Let $T: C_p(X) \to C_p(Y)$ be a linear continuous (even uniformly continuous) surjection. If $X$ is a metrizable compact space, then so is $Y$ (implicitly obtained in \cite{Usp}).
\item Metrizability is not preserved by linear homeomorphisms of both spaces $C_p(X)$ and $C_p^*(X)$ (apparently folklore, see \cite[Problem 263]{tk2}).   
\end{itemize}

More information can be found in 
\cite{ar1}, \cite{mar}, \cite{vanMill}, \cite{tk}, \cite{tk2}.

Everywhere below, we denote by $D(X)$ either $C^*(X)$ or $C(X)$, and $D_p(X)$ stands for $D(X)$ endowed with the pointwise convergence topology.
Let $E_p(X)$ and $E_p(Y)$ be dense linear subspaces of $D_p(X)$ and $D_p(Y)$, respectively. In this paper, our main interest is the following question.

\begin{que} \label{question}
Let $T: E_{p}(X)\to E_{p}(Y)$ be a linear continuous surjection. Under what conditions does $Y$ inherit a given topological property 
$\mathcal{P}$ from $X$?
\end{que}

A particular known positive result is the following.

\begin{pro} \label{prop_0} \cite[Proposition 2.1]{KawL} Let $X$ and $Y$ be compact metrizable spaces and
let $T: E_{p}(X)\to E_{p}(Y)$ be a linear continuous surjection, where $E_p(X)$ and $E_p(Y)$ are dense linear subspaces of $D_p(X)$ and $D_p(Y)$, respectively. 
If $\dim(X)=0$, then we have $\dim(Y)=0$.
\end{pro} 

The compactness assumption cannot be dropped in the above statement. The following example is a refinement of \cite[Example 2.2]{KawL}.

\begin{ex} \label{ex-0}
{\em Denote by $\J$ the zero-dimensional space of irrationals. Let $Y$ be any nonempty Polish space.
There exist a closed subset $X\subset \J$ and a continuous bijective mapping $\varphi$ from $X$ onto $Y$.
Consider the mapping $h$, dual to $\varphi$. We have that $h: C_p(Y) \to C_p(X)$ is a homeomorphic embedding of $C_p(Y)$ onto a dense linear subspace $E(X)$ of $C_p(X)$.
One can take a compact segment $Y= [0, 1]$, but $X$ is not compact.
Similarly, the linear continuous operator $T = h^{-1}: E(X) \to C_p(Y)$ naturally maps $E(X)$ onto $C_p(Y)$ but it does not preserve the property $\dim(X) = 0$.}
\end{ex}

Thus, there is no hope that, without extra assumptions, one can get a meaningful positive answer to Question \ref{question}.\\
\\
\begin{df}\label{df-correct}Let $E(X)$ be a linear subspace of $D(X)$.\\
\begin{enumerate}
\item[{\rm (a)}] We say that $E(X)$ is {\em correct} if for every $x\in X$ and its neighborhood $U$ there is $f\in E(X)$ with $f(x)\neq 0$ and $f(X\backslash U)=0$. 
\item[{\rm (b)}] A subset $E(X)$ is said to be {\em $\omega$-dense} (resp., {\em $\omega^*$-dense}) in $D(X)$ if for every countable set $K\subset X$, $f\in C(X)$ (resp., $f\in C^*(X)$) 
and $\varepsilon>0$ the set 
$U(f,K,\varepsilon)=\{g\in D(X): |f(x)-g(x)|<\varepsilon{~}\mbox{for all}{~}x\in K\}$\\
 meets $E(X)$. 
\end{enumerate}
\end{df}

\begin{ex}\label{ex-1}
{\em
\begin{enumerate}
\item[{\rm (a)}] Let $X$ be any space with a countable base $\{U_n\}$. For each pair $U_n, U_m$ such that $\cl(U_m) \subset U_n$ fix a continuous function $f_{n,m}: X \to [0,1]$ such that $f_{n,m}(X\backslash U_n)=0$ and $f_{n,m}(\cl(U_m))=1$. The linear span of all functions $f_{n,m}$ is a correct countable-dimensional linear subspace of $C_p^*(X)$. 
\item[{\rm (b)}] Obviously, every $\omega$-dense subset of $D(X)$ is dense in $C_p(X)$ and every $\omega^*$-dense subset of $C^*(X)$ is dense in $C_p^*(X)$. Every uniformly dense subspace of $C_p(X)$ ($C_p^*(X)$) is an $\omega$-dense subset of $C_p(X)$ (resp. $C_p^*(X)$). For instance, the linear subspace of piecewise linear functions in $C[a,b]$ is $\omega$-dense.
\end{enumerate}
}
\end{ex}
Everywhere below we say that a linear continuous surjection $T:E_p(X)\to E_p(Y)$ is {\em densely defined}
 if the domain $E(X)$ of $T$ is a correct linear subspace of $D(X)$ and the range $E(Y)$ is dense in $D_p(Y)$.

In the present paper we mainly focus on densely defined linear continuous operators $T: D_p(X)\to D_p(Y)$, where $X$ is a metrizable space and $Y$ is a perfectly normal space. 
Moreover, we consider all four possible cases when $D(X)$ is either $C(X)$ or $C^*(X)$ and $D(Y)$ is either $C(Y)$ or $C^*(Y)$.

Recall that a normal topological space $X$ is called {\em countable-dimensional} ({\em strongly countable-dimensional})
if $X$ can be represented as a countable union of normal finite-dimensional subspaces (resp., closed finite-dimensional subspaces).

A normal space is called a $C$-space if for every sequence $\{\mathcal G_i\}$ of open covers of $X$ there is a sequence $\{\mathcal H_i\}$ of families of pairwise disjoint open subsets of $X$ such that each member of $\mathcal H_i$ is contained in a member of $\mathcal G_i$ and $\bigcup_{i=1}^\infty\mathcal H_i$ is a cover of $X$. 
 
A normal space $X$ is called {\em weakly infinite-dimensional} if for every sequence $\{(A_n,B_n)\}_{n\geq 1}$ of pairs of disjoint closed subsets of $X$ there exist closed sets $L_n$, $n\geq 1$, such that $L_n$ is a partition between $A_n$ and $B_n$ and $\bigcap_{n=1}^\infty L_n=\varnothing$. 

It is well known that, in the class of metrizable spaces, every strongly countable-dimensional space is 
countable-dimensional, countable-dimensional spaces have property $C$ and $C$-spaces are weakly infinite-dimensional, but the reverse inclusions are not true, see \cite{en}. 

We consider properties $\mathcal P$ of metrizable spaces such that:
\begin{itemize}
\item [(a)] If $X\in\mathcal P$ and $F\subset X$ is closed, then $F\in\mathcal P$;
\item [(b)] If $X$ is a union of a $\sigma$-locally finite family $\mathcal A$ of closed sets in $X$ such that $A\in\mathcal P$ for all $A\in\mathcal A$, then $X\in\mathcal P$;
\item [(c)] If $f:X\to Y$ is a closed map and there is number $k$ such that $|f^{-1}(y)|\leq k$ for all $y\in Y$, then  $Y\in\mathcal P$ implies $X\in\mathcal P$.
\end{itemize}
Observe that condition $(b)$ is equivalent to the following two properties: $(b1)$ (The Countable Sum Theorem) If $X$ is a countable union of closed subsets each having the property $\mathcal P$,
 then $X\in\mathcal P$; $(b2)$ (The Locally Finite Sum Theorem) If $X$ is a union of a locally finite family of closed sets each having the property $\mathcal P$, then $X\in\mathcal P$.

Throughout the paper by dimension we mean the {\em covering dimension $\dim$}.
Classical results of dimension theory imply that {\em zero-dimensionality}, {\em countable-dimensionality},  {\em strong countable-dimensionality} and {\em weak infinite-dimensionality} satisfy conditions
 $(a) - (c)$ above, see \cite{en}. Moreover, {\em $\sigma$-compactness} and $C$-space property also satisfy these conditions. This is obvious for sigma-compactness. For property $C$, condition $(a)$ is also obvious, condition $(b1)$ was established in \cite{gv},  
condition $(b2)$ was established in \cite{ag} and condition $(c)$ follows from \cite{hy}.

We now formulate one of our main results.
Note that the particular case of Theorem \ref{theorem-main}, when $T:C_p(X)\to C_p(Y)$ is a linear continuous surjection, and both spaces $X$ and $Y$ are metrizable, 
was established in \cite[Theorem 1.6]{elv}, see Appendix.

\begin{thm}\label{theorem-main} Let $X$ and $Y$ be metrizable spaces and $\mathcal P$ be a topological property satisfying conditions 
$(a) - (c)$. Suppose that
$T:E_p(X)\to E_p(Y)$ is a densely defined linear continuous operator such that $E(Y)$ is $\omega^*$-dense in $D(Y)$.
 If all finite powers $X^k$, $k\in\mathbb N$, have the property $\mathcal P$ then all finite powers of $Y$ also have the property $\mathcal P$.
 \end{thm} 

An essential fact for our purposes is the following:  If $\mathcal P$ is either zero-dimensionality, strong countable-dimensionality or local compactness 
and $X$ is a metrizable space with $X\in\mathcal P$, then $X^k\in\mathcal P$ for all $k\in\mathbb N$. 

The particular case of Corollary \ref{cor-main} below, when $T:C_p(X)\to C_p(Y)$ is a linear continuous surjection, was established in \cite[Theorem 5.3]{Pelant}. 

\begin{cor}\label{cor-main}
Let $X$ and $Y$ be metrizable spaces and $T:E_p(X)\to E_p(Y)$ be a densely defined linear continuous operator. 
\begin{itemize}
\item[(i)] If $E(Y)$ is $\omega^*$-dense in $D(Y)$ and $X$ is $\sigma$-compact, then $Y$ is also $\sigma$-compact;
\item[(ii)] If $E(Y)$ is $\omega$-dense in $C(Y)$ and $X$ is compact, then $Y$ is compact.
\end{itemize}
\end{cor} 

In view of Example \ref{ex-0}, Theorem \ref{theorem-main} is no longer true for arbitrary dense linear subspaces $E(X)$ of $C_p(X)$ even if $T:E_p(X)\to C_p(Y)$ is a linear surjection.
An alternative explanation for why the assumption that $E(X)$ is a correct subspace of $D(X)$ is crucial to our work can be found in Remark \ref{r_key}.
Example \ref{ex-0} also shows that in Corollary \ref{cor-main} the assumptions that $E(Y)$ is $\omega^*$-dense or $\omega$-dense in $D(Y)$ and $C(Y)$, respectively, are essential. 

\begin{cor} \label{cor2}
Let $X$ and $Y$ be  metrizable spaces and let $T:E_p(X)\to E_p(Y)$ be a densely defined linear continuous operator such that $E(Y)$ is
 $\omega^*$-dense in $D_p(Y)$. 
If all finite powers of $X$ are weakly infinite-dimensional or have the property $C$, then all finite powers of $Y$ have the same property. 
\end{cor}
Note also that if $\mathcal P$ is zero-dimensionality or strong countable-dimensionality, then conditions $(c)$ and $(b1)$ are satisfied if $X$ is normal and $Y$ metrizable, 
see \cite[Theorem 3.3.10 and Theorem 5.4.5]{en}. Then we have the following version of Theorem \ref{theorem-main}. 

\begin{thm}\label{th9}
Let $X$ be a metrizable space, let $Y$ be perfectly normal and let $\mathcal P$ be either zero-dimensionality or strong countable-dimensionality.
 If there is a densely defined linear continuous operator 
 $T: E_p(X)\to E_p(Y)$ such that $E(Y)$ is $\omega^*$-dense in $D(Y)$, then $X\in\mathcal P$ implies $Y\in\mathcal P$.
\end{thm}
 
In the case where the range $E(Y)$ is only dense in $D_p(Y)$ (not necessarily $\omega^*$-dense in $D(Y)$), we have the following versions of Theorems \ref{theorem-main} and \ref{th9}.

\begin{thm} \label{th2} Let $K\subset Y$ be a countably compact closed set.
\begin{enumerate}
\item[{\rm (a)}] Let $X$ and $Y$ be metrizable spaces and $\mathcal P$ be a topological property satisfying conditions 
$(a) - (c)$. Suppose that
$T: E_p(X)\to E_p(Y)$ is a densely defined linear continuous operator. 
 If all finite powers $X^k$, $k\in\mathbb N$, have the property $\mathcal P$ then $K$ has the property $\mathcal P$.
\item[{\rm (b)}] Let $X$ be a metrizable space, let $Y$ be perfectly normal and let $\mathcal P$ be either zero-dimensionality or strong countable-dimensionality.
Suppose that
$T: E_p(X)\to E_p(Y)$ is a densely defined linear continuous operator. 
If all finite powers $X^k$, $k\in\mathbb N$, have the property $\mathcal P$ then $K$ has the property $\mathcal P$.
\end{enumerate}
\end{thm}

Zakrzewski \cite[Theorem 3.12]{kz} proved the following result: Suppose $X$ and $Y$ are $\sigma$-compact spaces such that $X$ is finite-dimensional.
If there is a linear continuous map $T:C_p(X)\to C_p(Y)$ such that $T(C_p(X))$ is dense in $C_p(Y)$ and $|\supp(l_y)|\leq m$ for every $y\in Y$, then
$\dim Y\leq m\cdot\dim X+m+m!-1$.  In the case when $X$ and $Y$ are metrizable, we can get a better estimate of $\dim Y$.

\begin{thm}\label{th11}
 Let $X$ be a metrizable space and $Y$ be a perfectly normal space. Suppose $T:E_p(X)\to E_p(Y)$ is a densely defined linear continuous operator such that $E(X)$ is $\omega^*$-dense in $D(Y)$.
If $|\supp(l_y)|\leq m$ for every $y\in Y$ then $\dim Y\leq m\cdot\dim X$.
\end{thm} 
 
In the last part of this paper, Section \ref{section4}, for arbitrary Tychonoff spaces $X$ and $Y$, we obtain new results assuming only that $T: E_p(X)\to E_p(Y)$ is a densely defined linear continuous operator.
We show that $X\in\mathcal P$ implies $Y\in\mathcal P$, where  $\mathcal P$ is the property $(\kappa)$, the strong $\sigma$-scatteredness,
or the property of being a $\Delta_1$-space. Also, we establish an analogous result for the property $\mathcal P$ of being a $\Delta$-space provided $E(X)$ is a strongly correct subspace of $D(X)$.
All necessary definitions are given in Section \ref{section4}.

These results strengthen several theorems from \cite{KL2}, \cite{KKL} and \cite{kk}. 

\section{Preliminary results}\label{section2}
In this section all topological spaces are assumed to be Tychonoff.

\begin{lem}\label{lem-0} 
Every correct subspace $E(X)\subset D(X)$ is dense in $D_p(X)$.
\end{lem}
\begin{proof}Indeed, let $f\in D(X)$ and
 $$U(f,K,\varepsilon)=\{g\in D(X): |f(x)-g(x)|<\varepsilon{~}\mbox{for all}{~}x\in K\}$$
  be a neighborhood of $f$ in $D_p(X)$, 
where $K=\{x_1,x_2,..,x_k\}\subset X$ is finite. We take disjoint neighborhoods $U_i$ of $x_i$ and real numbers $t_i\neq 0$ with $|f(x_i)-t_i|<\varepsilon$ for all $i$. Since $E(X)$ is a correct linear subspace of $D(X)$, for every $i$ there is $g_i\in E(X)$ with $g_i(x_i)=t_i$ and $g_i(x)=0$ for $x\not\in U_i$.
Then 
$$g=\sum_{i=1}^k g_i\in E(X)\,\, \mbox{and}\,\, g\in U(f,K,\varepsilon).$$
\end{proof}

Suppose $T:E_p(X)\to E_p(Y)$ is a linear continuous operator, where $E(X)$ is a linear subspace of $D(X)$.
Every $y\in Y$ generates a linear continuous functional $l_y: E_p(X)\to\mathbb R$ defined by $l_y(f)=T(f)(y)$. According to the analogue of the Hahn-Banach theorem for locally convex linear topological spaces \cite[Corollary 7.3.3]{nb}, $l_y$ can be continuously extended to a linear functional $\widetilde l_y:C_p(X)\to\mathbb R$. 
We denote by $\supp(l_y)$ the set of all $x\in X$ such that for every neighborhood $U_x$ of $x$ there exists $f\in E(X)$ with $f(X\backslash U_x)=0$ and $l_y(f)\neq 0$. Similarly, $\supp(\widetilde l_y)$ consists of all $x\in X$ such that for every neighborhood $U_x$ of $x$ there exists $f\in C(X)$ with $f(X\backslash U_x)=0$ and $\widetilde l_y(f)\neq 0$. 
It is well known, see for example \cite{ar1} or \cite{vanMill}, that for every $y\in Y$ there exist a finite set $\{x_1(y), x_2(y), \dots, x_k(y)\}$ in $X$ and real numbers 
$\lambda_i(y)\neq 0$, $i = 1, 2, \dots, k$,
such that for all $f\in C(X)$ we have
 $$\widetilde l_y(f)=\sum_{i=1}^k\lambda_i(y)f(x_i(y)).$$ 
 Obviously, $\supp(\widetilde l_y)=\{x_1(y), x_2(y), \dots, x_k(y)\}$. Here we provide some properties of the supports $\supp(l_y)$.

\begin{pro}\label{pro22} 
Let $T:E_p(X)\to E_p(Y)$ be a densely defined linear continuous operator. Then
\begin{itemize}
\item[(i)] $\supp(l_y)\neq\varnothing$ and $\supp(l_y)=\supp(\widetilde l_y)$ for every $y\in Y$;
\item[(ii)] For every $y\in Y$ there are real numbers $\lambda_i(y)$, $i=1,2,..,k$, such that $l_y(f)=\sum_{i=1}^k\lambda_i(y)f(x_i(y))$ for all $f\in E(X)$, where $\supp(l_y)=\{x_1(y),x_2(y),..,x_k(y)\}$;
\item[(iii)] If $\supp(l_y)\cap U\neq\varnothing$ for some open $U\subset X$, 
then there is a neighborhood $V\subset Y$ of $y$ such that $\supp(l_{z})\cap U\neq\varnothing$ for every $z\in V$.
\end{itemize}
\end{pro}
\begin{proof}
(i) Let us show first that $\supp(\widetilde l_y)\neq\varnothing$ for all $y\in Y$. Indeed, since $E(Y)$ is dense in $D_p(Y)$, there exists $g\in E(Y)$ with 
$g(y)\neq 0$. Then $T(f)=g$ for some $f\in E(X)$, so $\widetilde l_y(f)=g(y)$. Hence, $\widetilde l_y$ is a nontrivial linear continuous functional on $C_p(X)$
and $\supp(\widetilde l_y)\neq\varnothing$. 
To show the equality of both supports, let $\supp(\widetilde l_y)=\{x_1(y),x_2(y),..,x_k(y)\}$ and
 $$\widetilde l_y(f)=\sum_{i=1}^k\lambda_i(y)f(x_i(y))$$
 for every $f\in C(X)$, where all $\lambda_i(y)$ are non-zero real numbers. Striving for a contradiction, suppose there is $x_j(y)\in\supp(\widetilde l_y)\backslash\supp(l_y)$. Then, by the definition of $\supp(l_y)$, there exists a neighborhood $U$ of $x_j(y)$ such that
$l_y(f)=0$ for every $f\in E(X)$ with $f(X\backslash U)=0$. We can assume that $x_i(y)\not\in U$ for all $i\neq j$. Now, choose $f\in E(X)$ such that $f(x_j(y))\neq 0$ and $f(X\backslash U)=0$. Therefore, $\widetilde l_y(f)=l_y(f)=0$. On the other hand,
$$\widetilde l_y(f)=\lambda_jf(x_j(y))\neq 0.$$
 Thus, 
$\supp(\widetilde l_y)\subset\supp(l_y)$. The converse inclusion is obvious, so $\supp(\widetilde l_y)=\supp(l_y)$.

(ii) This item was actually proved in (i).
 
(iii) Suppose $x\in\supp(l_y)\cap U$. Then there exists $f\in E(X)$ such that $f(X\backslash U)=0$ and $l_y(f)\neq 0$. Since $l_y(f)=T(f)(y)$ and $T(f)\in C(Y)$, there is a neighborhood $V$ of $y$ with $T(f)(z)\neq 0$ for all $z\in V$. If $\supp(l_{z})\cap U=\varnothing$ for some $z\in V$, then $\supp(\widetilde l_{z})\cap U=\varnothing$. Hence 
$$\widetilde l_z(f)=\sum\lambda_i(z)f(x_i(z))=0.$$
Finally, $f\in E(X)$ implies $\widetilde l_z(f)=l_z(f)=T(f)(z)\neq 0$, a contradiction. 
\end{proof}

\begin{lem}\label{lem23}
Let $T: E_p(X)\to E_p(Y)$ be a densely defined linear continuous operator. If $K\subset X$ is a finite set then the cardinality of the set $P=\{y\in Y:\supp(l_y)\subset K\}$ does not exceed the size of $K$.
\end{lem}
\begin{proof}
This lemma was established by Krupski--Kucharski \cite[Lemma 1.2]{kk} when $T:C_p(X)\to C_p(Y)$ is a linear continuous surjection. The same proof works in our situation. For the reader's convenience we provide the details. Let $K=\{x_1,x_2,..,x_n\}$ and 
suppose $P$ contains a subset $P'$ consisting of $n+1$ different points $y_1,y_2,...,y_{n+1}$. Then 
consider the linear operator $\varphi:\pi(E(X))\to C_p(P')$ defined by $\varphi(\pi(f))(y_i)=T(f)(y_i)$, where $\pi:E(X)\to C(K)$ is the restriction operator. 
This operator is well-defined because $\supp(l_y)\subset K$ for all $y\in P$. Moreover, since $E(Y)$ is a dense subset of $D_p(Y)$, the image $\varphi(\pi(E(X)))$ is dense in $C_p(P')$. But $\pi(E(X))$ is at most an $n$-dimensional linear space as a subspace of $C(K)$, so the image $\varphi(\pi(E(X)))$ is of dimension $\leq n$.
 This is impossible because $C_p(P')$ is of dimension $n+1$ and does not contain a
 dense linear subspace of dimension $\leq n$.  
\end{proof}

\begin{rem}\label{r_key}
{\em Lemma \ref{lem23} has been used systematically throughout the paper. Let us explain here why the assumption that $E(X)$ is a correct subspace plays a key role. 
 First, correctness implies the density of $E(X)$ in $D_p(X)$ (Lemma \ref{lem-0}). But density itself is not enough to show that $\supp(l_y)=\supp(\widetilde l_y)$ for every $y\in Y$, 
see the proof of Proposition \ref{pro22}(i).

 In Lemma \ref{lem23} we show that the operator $\varphi$ is well-defined. Without the assumption that $E(X)$ is a correct subspace of $D(X)$, this cannot be established. 
}
\end{rem}

\begin{lem}\label{lem24}
Let $T:E_p(X)\to E_p(Y)$ be a densely defined linear continuous operator and $K\subset X$ and $L\subset Y$ 
such that $\supp(l_y)\subset K$ for all $y\in L$. Then there exists a densely embedded linear continuous operator $\varphi:E_p(K)\to E_p(L)$. 
\end{lem}
\begin{proof}
Let $\pi_K:E(X)\to D(K)$ and $\pi_L:E(Y)\to D(L)$ be the restriction maps 
$\pi_K(f)=f\restriction_{K}$ and $\pi_L(g)=g\restriction_{L}$.  
Since $E(X)$ is correct in $D_p(X)$ and $E(Y)$ is dense in $D_p(Y)$, one can show that
$E(K)=\pi_K(E(X))$ is a correct linear subspace of $D_p(K)$ and $E(Y)=\pi_L(E(Y))$ is a dense subspace of $D_p(L)$. 
Define $\varphi:E_p(K)\to E_p(L)$ by $\varphi(\pi_K(f))=\pi_L(T(f))$. This definition is correct because $\supp(l_y)\subset K$ for all $y\in L$. Moreover,
$\varphi$ is surjective. The continuity of $\varphi$ follows from the representation 
$$l_y(f)=\sum\lambda_i(y)f(x_i(y)), f\in E(X),$$ see Proposition \ref{pro22}(ii).
Indeed, if
$\pi_K(f_\alpha)$ is a net in $E_p(K)$ converging to some $\pi_K(f)$, then $l_y(\pi_K(f_\alpha))=\sum\lambda_i(y)f_\alpha(x_i(y))$. So,
$$\lim_\alpha l_y(\pi_K(f_\alpha))=l_y(f)=\pi_L(T(f))(y).$$
 Therefore, $\varphi$ is continuous. Linearity also follows from the equality
$$l_y(f)=\sum\lambda_i(y)f(x_i(y)), f\in E(X).$$
\end{proof}

Recall that a subset $L\subset Y$ is {\em bounded} if $f(L)$ is a bounded set in $\mathbb R$ for every $f\in C(Y)$. 
The next proposition is an analogue of Lemma 1.4.5 from \cite{bd}, see also \cite[Proposition 6.8.6]{vanMill}.

\begin{pro}\label{pro25}
Let $T:E_p(X)\to E_p(Y)$ be a densely defined linear continuous operator such that $E(Y)$ is $\omega$-dense in $C(Y)$. Then for every closed and bounded subset $K\subset X$, the set $L=\{y\in Y:\supp(l_y)\subset K\}$ is closed and bounded in $Y$.
\end{pro}
\begin{proof}
We adopt the proof of \cite[Proposition 6.8.6]{vanMill}. The conclusion of Proposition \ref{pro22}(iii)
is exactly what has been used in [22, Proposition 6.8.6] for the proof of the assertion stating that $L$ is closed in $Y$.

We prove the second assertion. On the contrary, if $L$ is not bounded, then $L$ contains a discrete and $C$-embedded subset $A=\{y_n:n\geq 1\}$ in $Y$.
 As we've already observed, for every $y_n$ there are numbers $\lambda_i(y_n)$ such that 
$l_{y_n}(f)=\sum_{i=1}^{k_n}\lambda_i(y)f(x_i(y_n))$, $f\in E(X)$, where $\supp(l_{y_n})=\{x_1(y_n),x_2(y_n),..,x_{k_n}(y_n)\}$. 
For each $n$ define the number $t_n$ in the following way: 
$$t_n=n\cdot\big(\sum_{i=1}^{k_n}|\lambda_i(y_n)|\big){~}\mbox{if}{~}\sum_{i=1}^{k_n}|\lambda_i(y_n)|>1{~}\mbox{and}{~}t_n=n{~}\mbox{if}{~}\sum_{i=1}^{k_n}|\lambda_i(y_n)|\leq 1.$$
 
Since $E(Y)$ is dense in $C_p(Y)$, for every $y_n$ there is a function $f_n\in E(X)$ such that $l_{y_n}(f_n)\neq 0$. 
Because $A$ is $C$-embedded in $Y$, there is $g\in C(Y)$ with $g(y_n)=t_n$ for each $n$. Next, use the assumption that $E(Y)$ is $\omega$-dense in $C(Y)$ to find $f\in E(X)$ with
$$T(f)(y_n)=\sum_{i=1}^{k_n}\lambda_i(y_n)f(x_i(y_n))\in (t_n-1,t_n+1).$$
Since $f(K)$ is bounded in $\mathbb R$, there is $c>0$ with $f(K)\subset[-c,c]$. Finally, take $n-1>c$. Then
$$t_n-1\leq T(f)(y_n)\leq\sum_{i=1}^{k_n}|\lambda_i(y_n)||f(x_i(y_n))|\leq c\cdot\sum_{i=1}^{k_n}|\lambda_i(y_n)|.$$
Therefore, if
$\sum_{i=1}^{k_n}|\lambda_i(y_n)|\leq 1$, we have 
$$t_n-1\leq c\cdot\sum_{i=1}^{k_n}|\lambda_i(y_n)|\leq c<n-1=t_n-1.{~}\mbox{This is a contradiction}.$$  

If $\sum_{i=1}^{k_n}|\lambda_i(y_n)|> 1$, then 
$$t_n-1\leq c\cdot\sum_{i=1}^{k_n}|\lambda_i(y_n)|\leq c<(n-1)\sum_{i=1}^{k_n}|\lambda_i(y_n)|=t_n-\sum_{i=1}^{k_n}|\lambda_i(y_n)|.$$  So,
$t_n-1<t_n-\sum_{i=1}^{k_n}|\lambda_i(y_n)|$ which implies $\sum_{i=1}^{k_n}|\lambda_i(y_n)|<1$, again a contradiction.
\end{proof}

\begin{cor}\label{cor26}
Let $T:E_p(X)\to E_p(Y)$ be a densely defined linear continuous operator such that $E(Y)$ is $\omega$-dense in $C(Y)$.
If $X$ is a pseudocompact space, then so is $Y$. 
\end{cor}

We say that a linear subspace $E\subset C^*(X)$ is {\em uniformly closed} if the limit of any uniformly converging sequence $\{f_n\}\subset E$ belongs to $E$.
Proposition \ref{pro27} below shows that a version of Corollary \ref{cor-main} remains true for arbitrary Tychonoff spaces $X$ and $Y$ provided $E(X)$
 is a correct and uniformly closed subspace of $C_p^*(X)$.
Note that, Proposition \ref{pro27} was established in \cite[Lemma 1.4.6]{bd} (see also \cite[Lemma 5.4]{Pelant}) when $X$ and $Y$ are metrizable spaces and $:C_p^*(X)\to C_p^*(Y)$ is a linear continuous surjection.

\begin{pro}\label{pro27}
Let $T:E_p(X)\to E_p(Y)$ be a densely defined linear continuous operator such that $E(X)$ is a correct uniformly closed subspace of $C^*(X)$ and $E(Y)$ is $\omega^*$-dense in $C^*(Y)$.
 Then for every compact metrizable set $K\subset X$, the set $L=\{y\in Y:\supp(l_y)\subset K\}$ is closed and bounded in $Y$.
\end{pro}
\begin{proof}
Let $E_u(X)$ be the space $E(X)$ with the topology of the uniform convergence. Since $E(X)$ is uniformly closed, $E_u(X)$ is a Banach subspace of $C_u^*(X)$. 
Consider $T$ as a linear operator between Banach spaces $T_u : E_u(X) \to C^*_u(Y)$.
Then the operator $T_u$ remains continuous according to the Closed Graph Theorem.
Hence  $||T_u||\leq m$ for some integer $m$.
By Proposition \ref{pro22}(ii), for every $y\in Y$ there are real numbers $\lambda_i(y)$ such that
$$T(f)(y)=\sum_{i=1}^{k}\lambda_i(y)f(x_i(y)){~}\mbox{for all}{~}f\in E(X),$$
 where $\supp(l_y)=\{x_1(y),..,x_k(y)\}$. We claim that 
$\sum_{i=1}^k|\lambda_i(y)|\leq m$ for all $y\in Y$. Indeed, take $f\in E(X)$ such that $||f||=1$ and $f(x_i(y))=sign \lambda_i(x_i(y))$ (such $f$ exists because $E(X)$ is a correct space). Then $|T(f)(y)|=T(f)(y)=\sum_{k=1}^k|\lambda_i(y)|\leq m$. Therefore, for all $y\in Y$ we have
$$\overline a(y)=\supp\{|l_y(f)|:f\in E{~}\mbox{and}{~}|f(x)|\leq 1{~}\mbox{for all}{~} x\in\supp(l_y)\}\leq m.$$  

Suppose $L$ is not bounded, so there is a discrete $C$-embedded subset $B=\{y_n\}_{n\geq 1}\subset L$ in $Y$. Let $\mathcal A$ be an uncountable almost disjoint family of infinite subsets of $\mathbb N$, i.e., $A_1\cap A_2$ is finite for all $A_1,A_2\in\mathcal A$. For every $A\in\mathcal A$ choose a function $g_A'\in C_p^*(Y)$ such that $g_A'(y_n)=m+1$ if $n\in A$ and $g_A'(y_n)=0$ if $n\not\in A$. Since $E(Y)$ is $\omega^*$-dense in $C^*(Y)$, for every $A\in\mathcal A$ there is $g_A\in E(Y)$ with $|g_A(x_n)-g_A'(x_n)|<1/2$, $n\in\mathbb N$. Next, choose $f_A\in E(X)$ such that $T(f_A)=g_A$. Since the family $\{f_A:A\in\mathcal A\}$ is uncountable and $C_u(K)$ is a separable Banach space, there are distinct $A_1,A_2\in\mathcal A$ with $||f_{A_1}|\restriction_K-f_{A_2}\restriction_K|| < 1$
 (here $f_{A_i}\restriction_K$ denote the restrictions of $f_{A_i}$ on $K$). Hence $|f_{A_1}(x)-f_{A_2}(x)|<1$ for all $x\in K$. Since $|T(f_{A_1}-f_{A_2})(y_n)|=|g_{A_1}(y_n)-g_{A_2}(y_n)|$ and $\overline a(y_n)\leq m$, $|g_{A_1}(y_n)-g_{A_2}(y_n)|\leq m$ for all $n$. On the other hand, if $n\in A_1\backslash A_2$, then $g_{A_1}(y_n)>m+1/2$ and $g_{A_2}(y_n)<1/2$. Hence, $|g_{A_1}(y_n)-g_{A_2}(y_n)|> m$ for all $n\in A_1\backslash A_2$, a contradiction.
\end{proof}

For any space $X$ and an integer $k$ let $[X]^k$ denote the space of all $k$-points subsets of $\overline X$ endowed with the Vietoris topology. It is well know that if $X$ is metrizable, then the Hausdorff's metric on $[X]^k$ generates the Vietoris topology of $[X]^k$, see for example \cite{mi}.  

\begin{lem}\label{lem28}
Let $X$ be a metrizable space such that all finite powers of $X$ have the property $\mathcal P$ and $k\in\mathbb N$ be a fixed integer, where $\mathcal P$ satisfies conditions $(a) - (c)$. Then all finite powers of $[X]^k$ also have the same property. 
\end{lem}
\begin{proof}
 Choose 
a $\sigma$-locally finite base $\mathcal B=\bigcup_{j=1}^\infty\mathcal B_j$, where each $\mathcal B_j$ is a locally finite family in $X$. Then for every $k$-tuple $(U_1,U_2,..,U_k)$ of elements of $\mathcal B$ with pairwise disjoint closures let $$W(U_1,U_2,..,U_k)=\{\{x_1,x_2,..,x_k\}\in [X]^k:x_i\in cl(U_i), i=1,2,..,k\}.$$ Every $W(U_1,U_2,..,U_k)$ is homeomorphic to 
$cl(U_1)\times cl(U_2)\times..\times cl(U_k)$. Indeed, let $[X]^{\leq k}$ be the space of all subsets of $X$ having at most $k$-points with the Vietoris topology, and consider the map $h:X^k\to [X]^{\leq k}$, defined by $h((x_1,x_2,..,x_k))=\{x_1,x_2,..,x_k\}$. According to \cite{mi}, $h$ is a continuous and closed surjection. Obviously, the restriction $h|\prod_{i=1}^kcl(U_i)$ is a closed bijection onto $W(U_1,U_2,..,U_k)$. So, it is a homeomorphism. 
Because $X^k\in\mathcal P$ and $\prod_{i=1}^kcl(U_i)$ is a closed subset of $X^k$,  condition $(a)$ implies that every $W(U_1,U_2,..,U_k)$ has the property $\mathcal P$. For every $k$-tuple $(j(1),j(2),..,j(k))$ denote by $\mathcal W_{j(1)j(2)..j(k)}$ the family of all 
$W(U_1,U_2,..,U_k)$ such that $U_i\in\mathcal B_{j(i)}$. Let us show that each $\mathcal W_{j(1)j(2)..j(k)}$ is a locally finite family in $[X]^k$. For any point $\widetilde x=\{x_1,x_2,..,x_k\}\in [X]^k$ take disjoint neighborhoods $V_i\in\mathcal B$ of $x_i$ in $X$, $i\leq k$, such that $V_i$ meets only finitely many elements of the family $\mathcal B_{j(i)}$. Then $W(V_1,V_2,..,V_k)$ is a neighborhood of $\widetilde x$ meeting only finitely many elements of
$\mathcal W_{j(1)j(2)..j(k)}$.
Therefore, the family $\mathcal W=\bigcup\{\mathcal W_{j(1)j(2)..j(k)}:(j(1),j(2),..,j(k)\in\mathbb N^k\}$ is a $\sigma$-locally finite cover of $[X]^k$ such that each element of $\mathcal W$ has the property $\mathcal P$.  Moreover, every product of $m$ elements of $\mathcal W$ is a closed subset of $X^{km}$, where $m$ is any integer. So, all finite products of elements of $\mathcal W$ have the property $\mathcal P$. On the other hand, since $\mathcal W$ is a $\sigma$-locally finite cover of $[X]^k$, every finite power $([X]^k)^m$ of 
$[X]^k$ has a $\sigma$-locally finite cover of elements which are products of $m$ many elements of $\mathcal W$. Therefore, all finite powers of $[X]^k$ have the property $\mathcal P$.
 \end{proof}
  The following three facts are well known: (i) $\dim X^k\leq k\cdot\dim X$, see \cite[Theorem 4.1.3 and Theorem 4.1.21]{en}; (ii) if a metrizable space $X$ has a $\sigma$-locally finite cover of closed sets each of dimension $\leq n$, then $\dim X\leq n$, see \cite[Theorem 3.1.15]{en}; if $\dim X\leq n$ and $F\subset X$ is closed, then $\dim F\leq n$. Using these three facts in the proof of 
	Lemma \ref{lem28}, we obtain the next statement.  
 \begin{cor}\label{cor29}
 Let $X$ be a finite-dimensional metrizable space and $k\in\mathbb N$ be a fixed integer. Then $\dim [X]^k\leq k\cdot\dim X$.
 \end{cor}

\section{Proofs of the results from the Introduction}\label{section3}
{\bf Proof of Theorem \ref{theorem-main}.} 
Suppose $T:E_p(X)\to E_p(Y)$ is a densely defined linear continuous operator such that $E(X)\subset D(X)$ is a correct subspace  and $E(Y)$ is $\omega^*$-dense in $D(Y)$.
\begin{claim}\label{claim1}
The sets  $Y_{k}=\{y\in Y: |\supp(l_y)|\leq k\}$ are closed in $Y$.  
\end{claim}
This claim follows from Proposition \ref{pro22}. Indeed, if $y\not\in Y_k$ for some $k$, then $\supp(l_y)$ contains at least $k+1$ points $x_1, x_2,\dots, x_{k+1}$ and 
let each $U_i$ be a neighborhood of $x_i$ with $U_i\cap U_j=\varnothing$ for $i\neq j$. Then by Proposition \ref{pro22}(iii), there exists a neighborhood $V$ of $y$ such that $\supp(l_z)\cap U_i\neq\varnothing$ for each $i$. So, $V\subset Y\backslash Y_k$.   

Since $Y$ is perfectly normal, each set $M_k=Y_k\backslash Y_{k-1}$ is a union of countably many closed subsets $M_{kn}$ of $Y$. Denote again
 $$a(y)=\supp\{|l_y(f)|:f\in E{~}\mbox{and}{~}|f(x)|<1{~}\mbox{for all}{~} x\in\supp(l_y)\}.$$
 The representation of $l_y$ from Proposition \ref{pro22}(ii) implies that $a(y)<\infty$ for every $y\in Y$. Now, consider the sets 
$$F_p=\{y\in Y: a(y)\leq p\}{~}\mbox{and}{~}F_{kp}=M_{k}\cap F_p, p\in\mathbb N.$$

\begin{claim}\label{claim2}
Every set $F_{kp}$ is closed in $M_k$.
\end{claim}
Suppose $y\in M_{k}\backslash F_{kp}$ and $\supp(l_y)=\{x_1(y),x_2(y),..,x_k(y)\}$. Then $a(y)>p$ and there is $f\in E(X)$ such that $|f(x_i(y))|<1$ for all $i$ and $|l_y(f)|>p$. Take disjoint neighborhoods $U_i$ of each $x_i(y)$ and let $J=\{i: f(x_i(y))\neq 0\}$. For every $i\in J$ there is $g_i\in E(X)$ such that $g_i(x_i(y))=f(x_i(y))$ and $g_i(Y\backslash U_i)=0$ (recall that $E(X)$ is a correct subspace of $D(X)$). Let $V_i=U_i$ if $i\not\in J$ and $V_i$ be an open subset of $U_i$  if $i\in J$ such that $|g_i(x)|<1$ for all $x\in V_i$. Then the function $g=\sum_{i\in J}g_i\in E(X)$ and has the following properties: 
$|g(x)|<1$ for all $x\in X$ and $g|\supp(l_y)=f|\supp(l_y)$. Hence, $|l_y(f)|=|l_y(g)|=|T(g)(y)|>p$ and let $W_1$ be a neighborhood of $y$ in $Y$ such that $|T(g)(z)|>p$ for all $z\in W_1$. According to Proposition \ref{pro22}(iii), there exists a neighborhood $W_2$ of $y$ in $Y$ such that
$\supp(l_z)\cap V_i\neq\varnothing$ for all $z\in W_2$ and $i=1,2,..,k$. Obviously, $W=W_1\cap W_2\cap M_k$ is a neighborhood of $y$ in $M_k$ and 
for every $z\in W$ we have $\supp(l_z)=\{x_1(z),x_2(z),..,x_k(z)\}$ with $x_i(z)\in V_i$ and $|l_z(g)|>p$. Moreover, $|g(x)|<1$ for all $x\in\supp(l_z)$ with
$z\in W$ implies $a(z)>p$. Therefore,  $F_{kp}$ is closed in $M_k$.

Claim 2 yields that each of the sets $M_{kpn}=F_{kp}\cap M_{kn}$ is closed in $Y$ and, obviously, $Y=\bigcup_{k,p,n\geq 1}M_{kpn}$.
For every $k\geq 1$ consider the maps  
$S_{kpn}:M_{kpn}\to [X]^k$, $S_{kpn}(y)=\supp(l_y)$. 
 It follows from Proposition \ref{pro22}(ii) that $S_{kpn}$ are continuous.
\begin{claim}\label{claim3}
All maps $S_{kpn}: M_{kpn}\to [X]^k$ are closed.
\end{claim}
Since $[X]^k$ is a metrizable space, it suffices to show that if $\{y_m\}$ is a sequence in $M_{kpn}$ and $S_{kpn}(y_m)$ converges in $[X]^k$, then
$\{y_m\}$ has an accumulation  point in $M_{kpn}$. Striving for a contradiction, suppose there is such a sequence $\{y_m\}\subset M_{kpn}$ without any accumulation point. Then $A=\{y_m:m\geq 1\}$ is a discrete subset of $Y$ and $S_{kpn}(y_m)$ converges to some $P=\{x_1,x_2,..,x_k\}\in [X]^k$. Hence, $S_{kpn}(y_m)=\{x_1(y_m),x_2(y_m),..,x_k(y_m)\}$  such that each of the sequences $\{x_i(y_m)\}_{m\geq 1}$ converges to $x_i$, $i=1,2,..,k$. So, 
$K=\bigcup_{i=1}^k\{x_i\}\cup\{x_i(y_m)\}_{m\geq 1}$ is a compact subset of $X$. According to Lemma \ref{lem24}, there is a densely defined linear continuous operator
$\varphi:\pi_K(E(X))_p\to\pi_A(E(Y))_p$, where $\pi_K:E(X)\to C(K)$ and $\pi_A:E(Y)\to C(A)$ are the restriction maps.
Let $\mathcal B$ be an uncountable almost disjoint family of infinite subsets of $\mathbb N$, i.e., $B_1\cap B_2$ is finite for distinct $B_1,B_2\in\mathcal B$. For any $B\in\mathcal B$ let $g_B'\in C_p^*(Y)$
be such that $g_B'(y_m)=p+1$ if $m\in B$ and $g_B'(y_m)=0$ if $m\not\in B$. Since $E(Y)$ is $\omega^*$-dense in $D(Y)$, there is $g_B\in E(Y)$ such that
$|g_B(y_m)-g_B'(y_m)|<1/3$ for all $B\in\mathcal B$ and $m\in\mathbb N$.
Then choose $f_B\in\pi(E)$ such that $\varphi(f_B)=g_B$. Put 
$s_B=(f_B(x_1),..,f_B(x_k))\in\mathbb R^k$. Since $\mathbb R^k$ is separable and $\mathcal B$ is uncountable, there are distinct $B(1),B(2)\in\mathcal B$ with 
$|f_{B(1)}(x_i)-f_{B(2)}(x_i)|<1$ for $i=1,2,..,k$. Because $f_{B(1)}$ and  $f_{B(2)}$ are continuous and the sequences $\{x_i(y_m)\}_{m\geq 1}$ converge to $x_i$, we can find $m_0$ such that $|f_{B(1)}(x)-f_{B(2)}(x)|<1$ for all $x\in S_{kpn}(y_m)$ and $m\geq m_0$. 
Then $\varphi(f_{B(1)}-f_{B(2)})=g_{B(1)}-g_{B(2)}$  and because $a(y_m)\leq p$, we have $|g_{B(1)}(y_m)-g_{B(2)}(y_m)|\leq p$ for all $m\geq m_0$.
On the other hand 
$B(1)\backslash B(2)$ is infinite and for every $m\in B(1)\backslash B(2)$ we have $g_{B(1)}(y_m)\geq p+2/3$ and $g_{B(2)}(y_m)\leq 1/3$. 
Hence, $|g_{B(1)}(y_m)-g_{B(2)}(y_m)|\geq p+1/3$ for some $m\geq m_0$. This contradiction completes the proof of Claim 3.
  
\begin{claim}\label{claim4}
The cardinality of all fibers of the map $S_{kpn}: M_{kpn}\to [X]^k$ is $\leq k$.
\end{claim}
Let $z\in M_{kpn}$, $S_{kpn}(z)=\{x_1(z), x_2(x), \dots, x_k(z)\}$ and $A(z)=\{y\in M_{kpn}:S_{kpn}(y)=S_{kpn}(z)\}$ is the fiber of $S_{kpn}$. Then,
by Lemma \ref{lem23}, the cardinality of the set $A(z)$ is $\leq k$.

To finish the proof of Theorem \ref{theorem-main} observe that 
$$Y=\bigcup\{M_{kpn}:n, k, p\in\mathbb N\}$$ 
such that each $M_{kpn}$ is a closed subset of $Y$. Let us show that all finite products $\prod_{i=1}^m M_{k_ip_in_i}\in\mathcal P$.
 To this end, observe that, according to Claims 3 and 4, all maps $S_{kpn}:M_{kpn}\to [X]^k$ are closed with cardinality of each fiber $\leq k$. So, all product maps 
$\prod_{i=1}^m S_{k_ip_in_i}:\prod_{i=1}^m M_{k_ip_in_i}\to\prod_{i=1}^m [X]^{k_i}$ are also closed having cardinality of each fiber $\leq k_1\cdot k_2\cdot...\cdot k_m$, see \cite[Theorem 3.7.9]{en1}. Since $\prod_{i=1}^m [X]^{k_i}\in\mathcal P$ (Lemma \ref{lem28}) and $\mathcal P$ satisfies condition $(c)$, all product of the form $\prod_{i=1}^m M_{k_ip_in_i}$ have the property $\mathcal P$.
Finally, because for every $m$ the power $Y^m$ is a union of countable many closed sets of the form $\prod_{i=1}^m M_{k_ip_in_i}$, by the Countable Sum Theorem, $Y^m\in\mathcal P$. Hence, all finite powers of $Y$ have the property $\mathcal P$.
\hfill $\Box$ 

{\bf Proof of Corollary \ref{cor-main}.} Since $\sigma$-compactness satisfies all conditions (a) - (c), Corollary \ref{cor-main}(i) follows from Theorem \ref{theorem-main}. Concerning Corollary \ref{cor-main}(ii), according to Corollary \ref{cor26}, $Y$ is pseudocompact. On the other hand, because $E(Y)$ is $\omega$-dense in $C(Y)$, it is also $\omega^*$-dense in $D(Y)$.
 So, by Theorem \ref{theorem-main}, $Y$ is also $\sigma$-compact. Therefore, $Y$ is compact. $\Box$  

{\bf Proof of Corollary \ref{cor2}.} Suppose both $X$ and $Y$ are metrizable spaces such that all finite powers of $X$ are weakly infinite- dimensional and $T:E_p(X)\to E_p(Y)$ is a densely defined linear continuous surjection. 
By Claim 1 the sets $Y_i=\{y\in Y:|supp(l_y)|\leq i\}$ are closed in $Y$.  Obviously, the sequence $\{Y_i\}$ is increasing and $Y=\bigcup_{i=1}^\infty Y_i$. Consider the linear operator $\pi_{i}\circ T:E_p(X)\to E_p(Y_{i})$, where $\pi_{i}:C_p(Y)\to C_p(Y_{i})$ is the restriction map. Since $\pi_{i}$ is continuous and linear surjection, $\pi_{i}\circ T$ is densely defined whose image is $\omega^*$-dense in $D(Y_i)$.
Because the weak infinite-dimensionality satisfies conditions $(a) - (c)$ and all finite powers of $X$ are weakly infinite-dimensional, according to Theorem \ref{theorem-main} all finite powers of any $Y_{i}$ are also weakly infinite-dimensional. Then $Y^m=\bigcup_{(i_1,i_2,..,i_m)}Y_{i_1}\times...\times Y_{i_m}$ for every $m$. Now, take $i_0$ such that 
$i_k\leq i_0$ for all $k=1,2,..,m$. 
Hence, $Y_{i_1}\times...\times Y_{i_m}$ is a closed subset of $Y_{i_0}^m$. Because 
$Y_{i_0}^m$ is weakly infinite-dimensional, so is the product $Y_{i_1}\times...\times Y_{i_m}$. Therefore, every finite power of $Y$ is a countable sum of weakly infinite-dimensional closed sets.
So, by \cite[Theorem 6.1.6]{en} all finite powers of $Y$ are also weakly infinite-dimensional. 
The case when all powers of $X$ have the property $C$ is similar. In this case we use \cite[Theorem 4.1]{gv} instead of \cite[Theorem 6.1.6]{en}. $\Box$

{\bf Proof of Theorem \ref{th9}.} We follow the proof of Theorem \ref{theorem-main}. Since in this case $Y$ is perfectly normal, every $M_k$ is a countable union of closed set 
$M_{kn}\subset Y$. So, the proof of Claims 2 and 3 remains valid. The proof of Claim 4 also works because all maps  $S_{kpn}: M_{kpn}\to [X]^k$ are continuous and closed such that $[X]^k$ is metrizable and $M_{kpn}$ are perfectly normal. Because all finite powers of $X$ have the property $\mathcal P$ and $\mathcal P$ satisfies conditions $(a) - (c)$, $[X]^k\in\mathcal P$ (see Lemma \ref{lem28}).
On the other hand, both zero-dimensionality and strong countable-dimensionality satisfy condition $(c)$ for closed maps with normal domains and metrizable ranges.
 So, all sets $M_{kpn}$ have the property $\mathcal P$.
Finally, by the Countable Sum Theorem for $\dim$ and strong countable-dimensionality \cite[Theorem 3.1.8 and Theorem 5.2.14]{en}, $Y\in\mathcal P$.

{\bf Proof of Theorem \ref{th2}.} We follow the notations from the proof of \ref{theorem-main}. Suppose $T:E_p(X)\to E_p(Y)$ is a densely defined linear continuous surjection, where $E(X)$ is a correct subspace of $D(X)$ and $E(Y)$ is a dense linear subspace of $D_p(Y)$. We have two possibilities, either $Y$ is metrizable and $\mathcal P$ is any property satisfying conditions $(a) - (c)$ or $Y$ is perfectly normal and $\mathcal P$ is zero-dimensionality or strong countable-dimensionality. In both cases $Y$ is perfectly normal and $Y$ is the union of the closed sets $M_{kpn}$. 
Then $K$ is also perfectly normal and all sets $K_{kpn}=M_{kpn}\cap K$ are closed in $K$ with $K=\bigcup_{k,p,n\geq 1}K_{kpn}$. Observe that in the first case $K$ is a compact metrizable space. Hence, in this case, if every  $K_{kpn}$ has the property $\mathcal P$, then so does $K$. In the second case $K$ is normal and both zero-dimensionality and strong countable-dimensionality satisfy the Countable Sum Theorem for normal spaces (see \cite[Theorem 3.1.8 and Theorem 5.2.14]{en}). Therefore, in both cases,
it suffices to show that each $K_{kpn}$ has the property $\mathcal P$. To this end, we consider the maps
 $$L_{kpn}:K_{kpn}\to [X]^k,\,\, \mbox{where}\,\, L_{kpn}(y)=\supp(l_y).$$ 
To show that all $L_{kpn}$ are closed maps, as in Claim 3, we need to show that if $\{y_m\}$ is a sequence in $K_{kpn}$ and $L_{kpn}(y_m)$ converges in $[X]^k$, then
$\{y_m\}$ has an accumulation  point in $K_{kpn}$. But in our case this is true because each $K_{kpn}$ is supposed to be a countably compact space, so the maps $L_{kpn}$ are closed.
 Moreover, by Lemma \ref{lem23}, the fibers of $L_{kpn}$ have cardinality $\leq k$.
 Then the arguments from the proof of Theorem \ref{theorem-main} and Theorem \ref{th9} ensure that each $K_{kpn}$ has the property $\mathcal P$. So, does $K$.
\hfill $\Box$ 

{\bf Proof of Theorem \ref{th11}}. Obviously, this is true if $\dim X=\infty$, so let $\dim X<\infty$. Following the notations from the proof of Theorem \ref{theorem-main}, 
for every $k=2,3,..,m$ 
 the sets $M_{kpn}=F_{kp}\cap M_{kn}$ are closed in $Y$ and 
$Y=\bigcup_{k=1}^m\bigcup_{p,n\geq 1}M_{kpn}$. By Claims 3-4, the maps $S_{kpn}: M_{kpn}\to [X]^k$ are perfect with finite fibers. According to Corollary \ref{cor29}, $\dim [X]^k\leq k\cdot\dim X$. Therefore, by the Dimension-Lowering Mapping Theorem \cite[Theorem 3.3.10]{en},
 $\dim M_{kpn}\leq k\cdot\dim X$ for all $n,p\geq 1$ and $k\leq m$.
So, by Countable Sum Theorem \cite[Theorem 3.1.8]{en}, $\dim Y\leq m\cdot\dim X$.

\section{Scattered-like properties $\mathcal P$}\label{section4}
In the last part of our paper we deal with several classes of Tychonoff topological spaces.
The arguments are independent of those which were developed in Section \ref{section3}.
We use repeatedly only Lemma \ref{lem23} which is true for all Tychonoff spaces.
Below we give relevant definitions, and summarize the relationships among these concepts.

\begin{df}\label{def:basic1} Let $X$ be a Tychonoff topological space.
\begin{itemize} 
    \item[\ding{192}] $X$ is called scattered if every subset $A$ of $X$ has an isolated (in $A$) point.
    \item[\ding{193}] $X$ is  strongly $\sigma$-scattered if $X$ is a union of countably many closed scattered subsets $X_n\subset X$.
		\item[\ding{194}] $X$ belongs to the class $\Delta_1$ if any sequence $\{A_n\}_{n\in \mathbb N}$ of disjoint countable subsets $A_n\subset X$ has a point-finite open expansion in $X$,
		i.e., there is a sequence of open sets $U_n\subset X$ such that $A_n\subset U_n$ for every $n\in \mathbb N$ and the family $\{U_n\}$ is point-finite in $X$.
		\item[\ding{195}] $X$ is said to have the property $(\kappa)$ if any sequence $\{K_n\}_{n\in \mathbb N}$ of disjoint finite subsets of $X$ has a subsequence $\{K_{n_k}\}$ admitting a point-finite open expansion. 
		\end{itemize}
\end{df}

The following implications hold: $\circled{1}$ $\Rightarrow$ $\circled{2}$ $\Rightarrow \circled{3} \Rightarrow \circled{4}$, and no implication can be reversed.
The class $\Delta_1$ was introduced and investigated in \cite{KKL}. In particular, $\circled{2} \Rightarrow \circled{3}$ follows from the results of  \cite{KKL}, 
because every scattered space is in $\Delta_1$, and the class $\Delta_1$ is invariant under countable unions of closed subspaces.
The property $(\kappa)$ was introduced by Sakai \cite{sa} who proved that $X$ has the property $(\kappa)$ if and only if $C_p(X)$ is $\kappa$-Fr\'echet-Urysohn.
		
\begin{thm}\label{th42}
Let $T:E_p(X)\to E_p(Y)$ be a densely defined linear continuous operator and $\mathcal P$ be either the strong $\sigma$-scatteredness or the property $(\kappa)$.
 Then $X\in\mathcal P$ implies $Y\in\mathcal P$.
\end{thm} 

\begin{proof}
This theorem was established by Krupski--Kucharski \cite{kk} in the case when $T:C_p(X)\to C_p(Y)$ is a linear continuous surjection. The Krupski--Kucharski's arguments work in our case by using 
Lemma \ref{lem23} instead of their Lemma 4.2 (actually it belongs to the mathematical folklore) and the fact that the support function $y\rightsquigarrow\supp(l_y)$ satisfies the condition from Proposition \ref{pro22}(iii). 
\end{proof}

The next result generalizes \cite[Theorem 3.16]{KKL}.

\begin{thm}\label{th43}
Let $T: E_p(X)\to E_p(Y)$ be a densely defined linear continuous operator. If $X\in\Delta_1$ then $Y\in\Delta_1$.
\end{thm}

\begin{proof}
It was shown in \cite{KKL} that $X\in\Delta_1$ is equivalent to the following property: for every sequence of points $\{x_n\}_{n\in\mathbb N} \subset X$ there is a point-finite open family $\{V_n\}_{n\in \mathbb N}$ in $X$ with $x_n\in V_n$ for all $n\in \mathbb N$.
We adopt the proof of \cite[Theorem 3.16]{KKL}. Let $\{y_n\}_{n\in\mathbb N}$ be a sequence in $Y$ and $A=\bigcup_{n\in\mathbb N}\supp(l_{y_n})$. Since $A$ is countable and $X\in\Delta_1$, there is a point-finite open expansion $\{U_x: x\in A\}$ of $A$. Then, according to the definition of $\supp(l_y)$, for every $n$ and $x\in\supp(l_y)$ there is $f_{x,n}\in E(X)$ such that $f_{x,n}(X\backslash U_x)=0$ and 
$l_{y_n}(f_{x,n})\neq 0$. Define 
$$\displaystyle V_n=\bigcap_{x\in\supp(l_{y_n})}\{y\in Y:l_{y}(f_{x,n})\neq 0\}.$$
 Each $V_n$ is open and $y_n\in V_n$. Let 
$$P_y=\{x\in A:\supp(l_y)\cap U_x\neq\varnothing\}, y\in Y.$$
 Since $\supp(l_y)$ is finite and the family $\{U_x: x\in A$ is point-finite, every $P_y$ is a finite set. Moreover,
$y\in V_n$ implies $\supp(l_{y_n})\subset P_y$. Indeed, if there is $x\in\supp(l_{y_n})\backslash P_y$, then $\supp(l_y)\cap U_x=\varnothing$. So, 
$f_{x,n}(\supp(l_y))=0$ which yields $l_y(f_{x,n})=0$. That contradicts $y\in V_n$. Finally, by Lemma \ref{lem23}, there are only finitely many $n$ such that 
$\supp(l_{y_n})\subset P_y$. Therefore, $\{V_n\}_{n\in\mathbb N}$ is a point-finite expansion of $\{y_n\}_{n\in\mathbb N}$.
\end{proof}

Now we turn our attention to the class of spaces $\Delta \subset \Delta_1$.
Recall that $X$ belongs to the class $\Delta$ if any sequence $\{X_n\}_{n\in \mathbb N}$ of disjoint subsets of $X$ has a point-finite open expansion in $X$ \cite{KL}.
It has been shown in \cite[Theorem 3.1]{KL2} that if there exists a linear continuous surjection $T: C_p(X) \to C_p(Y)$, where $X$ and $Y$ are arbitrary Tychonoff spaces, 
then $X\in \Delta$ implies $Y\in \Delta$. We observe that the actual proof of \cite[Theorem 3.1]{KL2} provides the same conclusion assuming that the range of 
operator $T$ is dense in $C_p(Y)$ (see \cite[Theorem 17]{FK}).

\begin{prob}\label{prob2} 
Let $T: E_p(X)\to E_p(Y)$ be a densely defined linear continuous operator. Is it true that $X\in\Delta$ implies $Y\in\Delta$?
\end{prob}

We give a partial positive answer to Problem \ref{prob2} under slightly stronger assumptions on the domain $E_p(X)$.
We say that $E(X)$ is {\em strongly correct} if for every $x\in X$ and its neighborhood $U$ there is $f\in E(X)$ with 
$f(X): X \to [0,1]$, $f(x)=1$ and $f(X\backslash U)=0$. 

\begin{thm}\label{th-D}
Let $T: E_p(X)\to E_p(Y)$ be a linear continuous operator such that $E(X)$ is a strongly correct linear subspace of $D(X)$ and $E_p(Y)$ is dense in $D_p(Y)$.
If $X\in\Delta$ then $Y\in\Delta$.
\end{thm}
\begin{proof}
It follows from \cite[Theorems 1.1 and 2.1]{KL} that $Y$ is a $\Delta$-space if and only if 
$C_{p}(Y)$ is a {\em large subspace} of $\R^{Y}$, i.e., for every mapping $g \in \R^Y$ there is a bounded set $A \subset C_p(Y)$ such that $g\in \cl_{\R^Y}(A)$.
By a {\em bounded set} $A$ in a topological vector space $V$ we understand any set which is absorbed by every $0$-neighborhood $V$, i.e., $A \subset \lambda V$ for some scalar $\lambda$.

The linear continuous surjective operator $T: E_p(X)\to E_p(Y)$ has a domain which is dense in $C_p(X)$ and a range which is dense in $C_p(Y)$.
Denote by  $\widehat{T}:\mathbb{R}^{X}\to \mathbb{R}^{Y}$ the linear continuous surjective extension of $T$ which is supplied by \cite[Lemma 3.2]{KL2}. 
Take arbitrary $g\in\mathbb{R}^{Y}$. There exists $f\in\mathbb{R}^{X}$ with $\widehat{T}(f)=g$. We will construct a bounded set $B\subset E_{p}(X)$ 
such that  $f\in \cl_{\R^X}(B)$. 

If  there exists a constant $r > 0$ such that $\sup\{|f(x)|: x \in X\} < r$,
 then we take $B=\{h \in E(X): \sup\{|h(x)|: x \in X\} < r \}$. It will be clear from the next part of the proof that $B$ is as required.

Let $f \in \R^X$ be unbounded. Denote by 
$$X_n =\{x \in X: n-1 \leq |f(x)| < n\}{~}\mbox{for each}{~}n \in \N.$$
Note that some sets $X_{n}$ might happen to be empty,
but the collection $\{X_{n}: n \in \N\}$ is a partition of $X$ with countably many nonempty sets $X_n$.
Define $\varphi: X \rightarrow \N$ by the rule: if $X_n \neq \emptyset$ then $\varphi(x) = n$ for every $x\in X_n$.
So, $|f| < \varphi$. 
By our assumption, $X\in\Delta$, hence there exists a point-finite open expansion $\{U_{n}:n\in\N\}$ of the partition $\{X_{n}:n\in\N\}$. 
 Define $F:X\rightarrow\N$ by $F(x)=\max\{n: x\in U_{n}\}$. Obviously, $f < F$. Finally, we define $B = \{h\in E(X): |h|\leq F\}$.
Then $B$ is a bounded set in $C_p(X)$. We claim that $f \in cl_{\R^X}(B)$, because for every finite subset $K \subset X$ there is a function $h \in B$
such that $f\restriction_{K} = h\restriction_{K}$.

 Indeed, given a finite subset $K \subset X$, let $\{V_{x}:x\in K\}$ be the family of pairwise disjoint open sets
 such that $x\in V_{x}\subset U_{\varphi(x)}$ and $V_x$ meets finitely many sets $U_n$ for every $x\in K$. Denote by $K^{\prime}$ the set of those $x\in K$ such that $f(x)\neq 0$.
 Since $E(X)$ is a strongly correct linear subspace of $D(X)$, for each $x\in K^{\prime}$, we can fix a continuous function $h_{x} \in E(X)$ such that
 $h_{x}:X\rightarrow [0,\varphi(x)]$, $h_{x}(x)= f(x)$ and $h_{x}$ is equal to the constant value $0$ on the closed set $X \setminus V_{x}$.
 One can verify that $h=\Sigma_{x\in K^{\prime}}h_{x} \in B$ is as required.
 
Finally, we define $A = T(B)$. It is easy to see that $A \subset E_p(Y) \subset C_p(Y)$ is bounded and
$f\in \cl_{\R^Y}(A)$ which means that $Y$ is a $\Delta$-space.
\end{proof}

Our last result provides a partial solution of \cite[Problem 3.13]{KL2}. 

\begin{thm}\label{unif} Let $X$ and $Y$ be Tychonoff spaces.
Assume that $C_p(X)$ is uniformly homeomorphic to $C_p(Y)$.
Then $X\in \Delta$ if and only if $Y\in \Delta$.
\end{thm}

\begin{proof}
The main idea repeats the previous proof. Let $T: C_p(X)\to C_p(Y)$ be a uniform homeomorphism. According to \cite[Theorem 3.1]{Arbit}, $T$ admits a surjective extension $\widehat{T}: \R^X \to \R^Y$ which remains a uniform homeomorphism. 

Take arbitrary $g\in\R^{Y}$. There exists $f\in\R^{X}$ with $\widehat{T}(f)=g$. Then there exists a bounded set $B\subset C_{p}(X)$ 
such that  $f\in \cl_{\R^X}(B)$. We define $A = T(B)$. It is easy to see that $A$ is bounded and 
$g\in \cl_{\R^Y}(A)$ which means that $C_{p}(Y)$ is a large subspace of $\R^{Y}$, equivalently, $Y$ is a $\Delta$-space. 
\end{proof}

\begin{rem}\label{r47} {\em The proof of Theorem \ref{unif} does not work if one assumes only that $T: C_p(X)\to C_p(Y)$ is a uniformly continuous surjective mapping.
The reason is the following: an extension to a uniformly continuous surjective mapping $\widehat{T}: \R^X \to \R^Y$ is not available.
For instance, there is a $1$-Lipschitz mapping from the plane $\R^2$ onto the plane without one point. }
\end{rem}

\section{Appendix}
It has been noted before that Theorem \ref{theorem-main} extends \cite[Theorem 1.6]{elv}, assuming that both spaces $X$ and $Y$ are metrizable.
Let us justify this claim. First, we need to refine some formulations appearing in \cite{elv}. 
In that paper, we considered properties $\mathcal P$ of normal spaces satisfying the following conditions:
\begin{itemize}
\item[(a)] If $X\in\mathcal P$ and $F\subset X$ is closed, then $F\in\mathcal P$;
\item[(b)] $\mathcal P$ is closed under finite products;
\item[(c)] If $X$ is a countable union of closed subsets each having the property $\mathcal P$, then $X\in\mathcal P$;
\item [(d)] If $f:X\to Y$ is a closed map with finite fibers, where $Y$ is a metrizable space with $Y\in\mathcal P$, then $X\in\mathcal P$.
\end{itemize}

Unfortunately, in the last paragraph of the proofs of \cite[Theorem 1.4]{elv} and \cite[Theorem 1.6]{elv}, 
we used the following incorrect statement: {\it If $X$ is a metrizable space and $q$ is any integer, then $[X]^q$ is homeomorphic
 to the set $W_q=\{(x_1,x_2,..,x_q)\in X^k: x_i\neq x_j{~}\hbox{for}{~}i\neq j\} \subset X^q$}.

Analysing our proofs from \cite{elv}, we observed that both \cite[Theorem 1.4]{elv} and \cite[Theorem 1.6]{elv} remain true if we make the following changes:
\begin{enumerate}
\item[1.] Replace item $(c)$ of the property $\mathcal P$ above with a stronger requirement $(c^{\prime})$: If $X$ is a union of a $\sigma$-locally finite family $\mathcal A$ of closed sets in $X$ such that
   $A\in\mathcal P$ for all $A\in\mathcal A$, then $X\in\mathcal P$ (exactly as item $(b)$ of the property $\mathcal P$ introduced in the current paper);
\item[2.] In order to prove that $[X]^q$ has the property $\mathcal P$, we can use current Lemma 2.9 instead of the aforementioned incorrect statement.
\end{enumerate}

It follows that, assuming both spaces $X$ and $Y$ are metrizable, \cite[Theorem 1.6]{elv} is indeed a particular case of Theorem \ref{theorem-main}. 

We should mention that the proofs of \cite[Theorem 4.8]{elv} and \cite[Theorem 4.9]{elv} also remain true because the properties of being a strongly $\sigma$-scattered space and
being a $\Delta_1$-space satisfy all required items $(a), (b), (c^{\prime}), (d)$, see \cite {KKL}.


\begin{thebibliography}{00}

\bibitem{ag}
D.~Addis and J.~Gresham, \textit{A class of infinite-dimensional spaces, Part I: Dimension theory and Alexandroff's Theorem}, Fund. Math. \textbf{101} (1978), 195--205.

\bibitem{Arbit} A. V. Arbit, 
\textit{The Lindel\"of number greater than continuum is $u$-invariant}, Serdica Math. J. \textbf{37} (2011), 143--162.

\bibitem{ar1} A.~Arkhangel'skii, 
\textit{Topological Function Spaces}, Kluwer Academic Publishers, Dordrecht, 1992.

\bibitem{bd} J.~Baars and J.~de Groot,
 \textit{On topological and linear equivalence of certain function spaces}, CWI tract 86, Center for for Mathematics and Computer Science, Amsterdam, 1992.

\bibitem{Pelant} J.~Baars, J.~de Groot and J. Pelant,
\textit{Function spaces of completely metrizable spaces}, Trans. Amer. Math. Soc. \textbf{340} (1993), 871--883.

\bibitem{en} R.~Engelking, 
\textit{Theory of dimensions, finite and infinite}, Sigma Series in Pure Mathematics, 10. Heldermann Verlag, 1995.

\bibitem{en1} R.~Engelking, 
\textit{General Topology},
Heldermann Verlag, Berlin, 1989.

\bibitem{elv} A.~Eysen, A.~Leiderman and V.~Valov,
 \textit{Linear and uniformly continuous surjections between $C_p$-spaces over metrizable spaces}, 
Math. Slovaca \textbf{75} (2025), 669--678.

\bibitem{ev}
A.~Eysen and V.~Valov, \textit{On uniformly continuous surjections between function spaces}, Canad. J. Math. (published online Sept 25, 2025),\\
 https://doi.org/10.4153/S0008414X25101636

\bibitem{FK} J. C.~Ferrando, J. K\c akol,
\textit{Distinguished dense $C_p$-subspaces},
Topology Appl. \textbf{377} (2026) 109647.

\bibitem{gu}
S.~Gul'ko, \textit{On uniform homeomorphisms of spaces of continuous functions}, Trudy Mat. Inst. Steklov, \textbf{193} (1992), 82--88 (in Russian):
 English translation: Proc. Steklov Inst. Math. \textbf{193} (1992), 87--93. 
 
\bibitem{gv}
V.~Gutev and V.~Valov, \textit{Continuous selections and $C$-spaces}, Proc. Amer. Math. Soc.  \textbf{130} (2002), 233--242.

\bibitem{hy}
Y.~Hattori and K.~Yamada, \textit{Closed pre-images of $C$-spaces}, Math. Japonica \textbf{34} (1989), 555--561.

\bibitem{KL} J. K\c akol and A. Leiderman,
\textit{A characterization of $X$ for which spaces $C_p(X)$ are distinguished and its applications},
Proc. Amer. Math. Soc., series B,  \textbf{8} (2021), 86--99.

\bibitem{KL2} J. K\c akol and A. Leiderman,
\textit{Basic properties of $X$ for which the space $C_p(X)$ is distinguished},
 Proc. Amer. Math. Soc., series B,  \textbf{8} (2021), 267--280.

\bibitem{KKL} J. K\c akol, O. Kurka and A. Leiderman,  
 \textit{Some classes of topological spaces extending the class of $\Delta$-spaces},
 Proc. Amer. Math. Soc. \textbf{152} (2024), 883--899.

\bibitem{KawL} K.~Kawamura and A.~Leiderman,
 \textit{Linear continuous surjections of $C_p$-spaces over compacta},
 Topology Appl. \textbf{227} (2017), 135--145. 

\bibitem{kk} M.~Krupski and K.~Kucharski,
 \textit{On the $A$-invariance of the hereditary Baire property and related results}, preprint,\\
arxiv.org/abs/2608.11115

\bibitem{lmp} A.~Leiderman, S.~Morris and V.~Pestov,
 \textit{The free abelian topological group and the free locally convex space on the unit interval},
 J. London Math. Soc. \textbf{56} (1997), 529--538.

\bibitem{mar} W. Marciszewski,
 \textit{Function spaces}, 
in M. Hu\v sek and J. van Mill eds., Recent progress in general topology II, Elseiver (2002), 345--369.

\bibitem{mi}
E.~Michael, \textit{Topologies on spaces of subsets}, Trans. Amer. Math. Soc. {\bf 71} (1951), 152--182.

\bibitem{vanMill} J. van Mill,
\textit{The Infinite-Dimensional Topology of Function Spaces}, North-Holland Mathematical Library 64, North-Holland, Amsterdam, 2001.

\bibitem{nb} L. Narici and E. Beckenstein, 
\textit{Topological Vector Spaces, Second Edition}, CRC Press, Taylor and Francis Group, 2010.

\bibitem{p} V.~Pestov,
 \textit{The coincidence of the dimension $\dim$ of $l$-equivalent topological spaces}, Soviet Math. Dokl.
 \textbf{26} (1982), 380--383. 

\bibitem{sa} M.~Sakai,
 \textit{Two properties of $C_p(X)$ weaker than the Fr\'echet-Urysohn property}, Topology Appl. \textbf{153} (2006), 2795--2804.

\bibitem{tk} V. V.~Tkachuk,
 \textit{$C_p$-theory problem book. Topological and function spaces. Problem Books in Mathematics}, Springer, Berlin, New York, 2011.

\bibitem{tk2} V. V. ~Tkachuk,
 \textit{$C_p$-theory problem book. Functional equivalencies. Problem Books in Mathematics}, Springer, Berlin, New York, 2016.

\bibitem{Usp} V. Uspenskii,
\textit{A characterization of compactness in terms of the uniform structure in a function space},
Russian Math. Surveys, \textbf{37(4)} (1982), 143--144.

\bibitem{kz} K.~Zakrzewski,
 \textit{Function spaces on Corson-like compacta}, Results Math. \textbf{80:75} (2025), 21 pp.


\end{thebibliography}
\end{document}